\documentclass[11pt]{amsart}
\usepackage[centertags]{amsmath}
\usepackage{amsfonts}
\usepackage{amssymb}
\usepackage{amsthm}
\usepackage{newlfont}
\usepackage{amscd}
\usepackage{mathrsfs}
\usepackage[all,cmtip]{xy}
\usepackage{tikz}
\usetikzlibrary{trees}
\usepackage[pagebackref=false,colorlinks]{hyperref}
\hypersetup{pdffitwindow=true,linkcolor=blue,citecolor=blue,urlcolor=cyan}
\usepackage{cite}
\usepackage{fancyhdr}
\usepackage{stmaryrd}

\usepackage[OT2,OT1]{fontenc}
\newcommand\cyr{%
\renewcommand\rmdefault{wncyr}%
\renewcommand\sfdefault{wncyss}%
\renewcommand\encodingdefault{OT2}%
\normalfont
\selectfont}
\DeclareTextFontCommand{\textcyr}{\cyr}

\newcommand{\mint}{{\times}\kern-0.89em{\int}} 

\usepackage[toc, page] {appendix}

\newcommand{\ks}{\boldsymbol{\kappa}}

\newcommand{\sha}{\textrm{{\cyr SH}}}

\newtheorem{thm}{Theorem}[section]
\newtheorem{cor}[thm]{Corollary}
\newtheorem{lem}[thm]{Lemma}

\newtheorem{assu}[thm]{Assumption}
\newtheorem{conj}[thm]{Conjecture}

\theoremstyle{definition}

\newtheorem{rem}[thm]{Remark}

\begin{document}
\title{Indivisibility of Heegner points and arithmetic applications}
\author{Ashay A. Burungale}
\address[Burungale]{School of Mathematics, Institute for Advanced Study, Einstein Drive, Princeton, NJ 08540, USA}
\email{ashayburungale@gmail.com}
\author{Francesc Castella}
\address[Castella]{Department of Mathematics, Princeton University, Washington Road, Princeton, NJ 08544-1000, USA}
\email{fcabello@math.princeton.edu}
\author{Chan-Ho Kim}
\address[Kim]{School of Mathematics, Korea Institute for Advanced Study, 85 Hoegi-ro, Dongdaemun-gu, Seoul 02455, Republic of Korea}
\email{chanho.math@gmail.com}
\date{\today}
\maketitle
\begin{abstract}
We upgrade Howard's divisibility towards Perrin-Riou's Heegner point main conjecture to the predicted equality. Contrary to previous works in this direction, our main result allows for the classical Heegner hypothesis and non-squarefree conductors. The main ingredients we exploit  
are W.~Zhang's proof of Kolyvagin's conjecture, Kolyvagin's structure theorem for Shafarevich--Tate groups, and the explicit reciprocity law for Heegner points.
%We discuss the application of the indivisibility of derived Heegner points \`{a} la Wei Zhang to the main conjecture for Heegner points \`{a} la Perrin-Riou. As a result, we upgrade the divisibility via the Heegner point Kolyvagin systems \`{a} la Howard to the predicted equality for rank one elliptic curves of not necessarily square-free conductor.
\end{abstract}
\setcounter{tocdepth}{1}
\tableofcontents

\section{Introduction}
%\subsection{Heegner point main conjecture}
Let $E/\mathbb{Q}$ be an elliptic curve of conductor $N$
and let $K$ be an imaginary quadratic field of discriminant $D_K$ with $(D_K, N) = 1$.
Then $K$ determines a factorization
\[
N=N^+N^-
\]
with $N^+$ (resp. $N^-$) divisible only by primes which are split (resp. inert) in $K$.
Throughout this paper, the following hypothesis will be in force:
\begin{assu}[Generalized Heegner hypothesis] \label{assu:gen_heeg}
	$N^-$ is the square-free product of an \emph{even} number of primes. 
\end{assu}
Let $p>3$ be a good ordinary prime for $E$ with $(p, D_K) = 1$, and 
let $K_\infty$ be the anticyclotomic $\mathbb{Z}_p$-extension of $K$. 
%Let $\iota:\Lambda\rightarrow\Lambda$ be the involution induced by the inversion in $\mathrm{Gal}(K_\infty/K)$.
%Set $G_F:=\mathrm{Gal}(\overline{\mathbf{Q}}/F)$ for any number field $F$. 
Under Assumption \ref{assu:gen_heeg}, the theory of complex multiplication provides a collection of CM points on a Shimura curve with ``$\Gamma_0(N^+)$-level structure'' attached to the quaternion algebra $B/\mathbf{Q}$ of discriminant $N^-$ defined over ring class extensions of $K$. By modularity, these points give rise to Heegner points on $E$ defined over the ring class extensions, and exploiting the $p$-ordinarity assumption these can be turn into a norm-compatible system of Heegner points on $E$ over the $\mathbb{Z}_p$-extension $K_\infty/K$.

Let $T$ be the $p$-adic Tate module of $E$, and set $V:=T\otimes\mathbb{Q}_p$ and $A:= V/T \simeq E[p^\infty]$. Let $\Lambda = \mathbf{Z}_p \llbracket \mathrm{Gal}(K_\infty/K )\rrbracket$ be the anticyclotomic Iwasawa algebra, and let $\mathbf{T}$ and $\mathbf{A}$ be the $\Lambda$-adic versions of $T$ and $A$, respectively, recalled in Section \ref{sec:Sel}.
Let
\[
\mathrm{Sel}_{\mathrm{Gr}}(K,\mathbf{T})\subset \varprojlim\mathrm{H}^1(K_n,T),\quad \mathrm{Sel}_{\mathrm{Gr}}(K,\mathbf{A})\subset \varinjlim\mathrm{H}^1(K_n,A)
\] 
be $\Lambda$-adic Greenberg ordinary Selmer groups defined in \cite{howard-kolyvagin, howard-gl2-type} (and recalled in Section~\ref{sec:Sel} below), where $K_n$ is the unique subextension of $K_\infty$ with $[K_n:K]=p^n$. 
Letting $\mathrm{Sel}_{p^m}(E/K_n)\subset\mathrm{H}^1(K_n,E[p^m])$ be the $p^m$-th descent Selmer groups fitting into the fundamental exact sequence
\[
0\longrightarrow E(K_n)\otimes\mathbf{Z}/p^m\mathbf{Z}\longrightarrow\mathrm{Sel}_{p^m}(E/K_n)\longrightarrow\sha(E/K_n)[p^m]\longrightarrow 0,
\]
there are $\Lambda$-module pseudo-isomorphisms
\[
\mathrm{Sel}_{\mathrm{Gr}}(K,\mathbf{T}) \sim \varprojlim_n\varprojlim_m\mathrm{Sel}_{p^m}(E/K_n),\quad \mathrm{Sel}_{\mathrm{Gr}}(K,\mathbf{A}) \sim \varinjlim_n\varinjlim_m\mathrm{Sel}_{p^m}(E/K_n).
\]

Via the Kummer map, the norm-compatible sequence of Heegner points on $E$ along $K_{\infty}/K$ gives rise to a $\Lambda$-adic Heegner cohomology class $\kappa_{1}^{\infty} \in\mathrm{Sel}_{\mathrm{Gr}} \left( K, \mathbf{T} \right)$ which was first shown to be non-torsion by Cornut--Vatsal \cite{cornut-vatsal}.  
%In view of the fundamental result in \cite{cornut-vatsal}, the class $\kappa_{1}^\infty$ is $\Lambda$ non-torsion. 
After Kolyvagin \cite{kolyvagin-mw-sha}, the non-triviality of a Heegner point over a ring class field $H/K$ implies that the Mordell--Weil rank of the underlying abelian variety over $H$ being one and also the finiteness of the corresponding Tate--Shafarevich group, with the index of the Heegner point in the Mordell--Weil group being closely to the size of the Tate--Shafarevich group (essentially, the latter is the square of the former). After Perrin-Riou \cite[Conj.~B]{perrin-riou-heegner} and Howard \cite{howard-gl2-type}, a $\Lambda$-adic analogue of this result takes the form of the following  ``Heegner point main conjecture'', where we let $\iota:\Lambda\rightarrow\Lambda$ be the involution induced by the inversion in $\mathrm{Gal}(K_\infty/K)$.

\begin{conj}[Perrin-Riou, Howard]\label{conj:HPMC}
	The $\Lambda$-modules $\mathrm{Sel}_{\mathrm{Gr}}(K,\mathbf{T})$ and  $\mathrm{Sel}_{\mathrm{Gr}}(K,\mathbf{A})$ have rank $1$ and corank 1, respectively. Letting
	\[
	X=\mathrm{Hom}_{\mathbf{Z}_p}(\mathrm{Sel}_{\mathrm{Gr}}(K,\mathbf{A}),\mathbf{Q}_p/\mathbf{Z}_p)
	\]
	be the Pontrjagin dual of $\mathrm{Sel}_{\mathrm{Gr}}(K,\mathbf{A})$, there is a torsion $\Lambda$-module $M_\infty$ with: 
	\begin{enumerate}
		\item $\mathrm{char}(M_\infty)=\mathrm{char}(M_\infty)^\iota$, 
		\item $X\sim\Lambda\oplus M_\infty\oplus M_\infty$,
		\item $\mathrm{char}(M_{\infty})=\mathrm{char}\biggl( \dfrac{ \mathrm{Sel}_{\mathrm{Gr}} \left( K, \mathbf{T} \right) }{ \Lambda\kappa^\infty_1 }\biggr)$.	
	\end{enumerate}
\end{conj}
The third statement is a form of Iwasawa main conjecture involving zeta elements which similarly appears in other settings, notably \cite[Thm.~5.1]{rubin-appendix} and \cite[Conj.~12.10]{kato-euler-systems}.
Here, we assume that the Manin constant is 1 and that $\mathcal{O}^\times_K = \lbrace \pm 1 \rbrace$ for notational simplicity.
\subsection{Main result} 

%In this paper we prove new cases of Conjecture~\ref{conj:HPMC}, %(especially for non-semistable elliptic curves), 
%complementing the divisibility established by Howard  \cite{howard-kolyvagin, howard-gl2-type}.  
Similarly as in \cite[Notations]{wei-zhang-mazur-tate}, we consider the following condition on the triple $(E,p,K)$:

\begin{assu}[Condition $\heartsuit$]\label{assu:heart}
	Denote by $\mathrm{Ram}(\overline{\rho})$ the set of primes $\ell$ dividing exactly $N$ such that the $G_{\mathbf{Q}}$-module $E[p]$ is ramified at $\ell$. Then:
	\begin{enumerate}
		\item $\mathrm{Ram}(\overline{\rho})$ contains all primes $\ell\Vert N^+$. 
		\item $\mathrm{Ram}(\overline{\rho})$ contains all primes  $\ell\vert N^-$. %such that $\ell\equiv\pm 1 \pmod{p}$. 
		\item If $N$ is not square-free, then $\#\mathrm{Ram}(\overline{\rho})\geqslant 1$, and either $\mathrm{Ram}(\overline{\rho})$ contains a prime $\ell\vert N^-$ or there are at least two primes $\ell\Vert N^+$.
		\item If $\ell^2\vert N^+$, then $\mathrm{H}^0(\mathbf{Q}_\ell, \overline{\rho}) = 0$.
	\end{enumerate}
\end{assu}

\begin{rem} \label{rem:conditionCR}
	This is a slight strengthening of Condition~$\heartsuit$ in \cite{wei-zhang-mazur-tate}, where in part (2) $E[p]$ is only required to be ramified at the primes $\ell\vert N^-$ with $\ell\equiv\pm{1}\pmod{p}$. 
\end{rem}

Under Condition~$\heartsuit$ (and other hypotheses recalled in Theorem~\ref{thm:indivisibility} below), W.~Zhang \cite{wei-zhang-mazur-tate} has recently obtained a proof of Kolyvagin's conjecture \cite{kolyvagin-selmer}. 
Concerning the nature of this conjecture, let us just mention here that it concerns the $p$-indivisibility of so-called derived Heegner classes on $E$, and as such it does not seem to have an Iwasawa-theoretic flavour. 

In this note, we shall build on W.~Zhang's result to prove the following theorem towards Conjecture~\ref{conj:HPMC}, where in addition to Assumptions~\ref{assu:gen_heeg} and \ref{assu:heart}, the following is in force:

\begin{assu}%[Split] 
	\label{assu:split}
	\hfill
	\begin{enumerate}
		\item $p = \mathfrak{p}\overline{\mathfrak{p}}$ splits in $K$.
\item $ \overline{\rho}: G_K\rightarrow\mathrm{Aut}_{\mathbf{F}_p}(E[p])$ is surjective.
	\end{enumerate}
\end{assu}

\begin{thm}[Main result]\label{thm:main}
Suppose that the triple $(E, K ,p)$ satisfies Assumptions \ref{assu:gen_heeg}, \ref{assu:heart}, and \ref{assu:split}, and assume in addition that $\mathrm{ord}_{s=1}L(E/K,s)=1$.
Then Conjecture~\ref{conj:HPMC} holds.
\end{thm}

\subsection{Outline of the proof} 
After Perrin-Riou's work, the first results towards Conjecture~\ref{conj:HPMC} were due to Bertolini \cite{bertolini} and Howard \cite{howard-kolyvagin, howard-gl2-type}, which under mild hypotheses established one of the divisibilities predicted by the third statement in the conjecture. More precisely, adapting to the anticyclotomic setting the Kolyvagin system machinery of Mazur--Rubin \cite{mazur-rubin-book}, Howard constructed a \emph{$\Lambda$-adic Kolyvagin system}  $\boldsymbol{\kappa}^\infty$ whose base class $\kappa_{1}^\infty$ could be shown to be non-trivial by Cornut--Vatsal \cite{cornut-vatsal}, yielding a proof of all the statements in Conjecture~\ref{conj:HPMC} \emph{except} for the divisibility ``$\subseteq$'' in the third part. 

Later, the first cases of the full Conjecture~\ref{conj:HPMC} were obtained in \cite[Thm.~1.2]{wan-heegner} and \cite[Thm.~3.4]{castella-beilinson-flach}. These were obtained by building on X.~Wan's work \cite{wan-rankin-selberg}, which when combined with the reciprocity law for Heegner points \cite{cas-hsieh1} yields a proof of the missing divisibility ``$\subseteq$''. However, these results excluded the case $N^-=1$ (i.e. the ``classical'' Heegner hypothesis) and $N$ is assumed to be square-free. In contrast, our proof of Theorem~\ref{thm:main} is based on a  different idea, dispensing with the use of \cite{wan-rankin-selberg} and allowing for those excluded cases. Moreover, we expect the analytic rank $1$ hypothesis made in Theorem~\ref{thm:main} to not be essential to our method  (see Remark~\ref{rem:higher-rank}).

As alluded to above, Howard's results in \cite{howard-kolyvagin, howard-gl2-type} 
%Howard established one of the divisibilities in Conjecture~\ref{conj:HPMC} 
are based on the Mazur--Rubin machinery of Kolyvagin systems, suitably adapted to the anticyclotomic setting. As essentially known to Kolyvagin \cite{kolyvagin-selmer}, %(in the context of Heegner points), 
the upper bound  %on Selmer groups 
provided by this machinery can be shown to be \emph{sharp} under a certain nonvanishing hypothesis;  in the framework of \cite{mazur-rubin-book}, this corresponds to the Kolyvagin system being \emph{primitive} \cite[Def.~4.5.5 and 5.3.9]{mazur-rubin-book}. 

Even though primitivity was not incorporated into Howard's %anticyclotomic 
treatment \cite{howard-kolyvagin, howard-gl2-type}\footnote{See \cite{howard-bipartite} however, esp. Theorem~3.2.3, 
	although we will have no use for any of the results in that paper.}, we shall upgrade his divisibility to an equality 
%(therefy proving Theorem~\ref{thm:main}) 
by building on W.~Zhang's proof of Kolyvagin's conjecture \cite{wei-zhang-mazur-tate}. In order to carry out this strategy, we consider a different anticyclotomic main conjecture. Under Assumption~\ref{assu:split},  Bertolini--Darmon--Prasanna \cite{bertolini-darmon-prasanna-duke} (as extended by Brooks \cite{brooks} for $N^-\neq 1$) have constructed a $p$-adic $L$-function $\mathscr{L}_{\mathfrak{p}}^{BDP}\in\Lambda^{\mathrm{ur}}:=\mathbf{Z}^{\mathrm{ur}}_p\hat\otimes\Lambda$, 
where $\mathbf{Z}^{\mathrm{ur}}_p$ is the completion of the maximal unramified extension of $\mathbf{Z}_p$, $p$-adically interpolating a \emph{square-root} of certain Rankin--Selberg $L$-values. A variant of Greenberg's main conjectures \cite{greenberg-motives} relates $\mathscr{L}_{\mathfrak{p}}^{BDP}$ to the characteristic ideal of an ``$N^-$-minimal'' anticyclotomic Selmer group
\[
\mathrm{Sel}^{N^-}_{\emptyset,0}(K,\mathbf{A})\subset \varinjlim\mathrm{H}^1(K_n,A)
\]
defined in \cite[$\S$2.3.4]{jetchev-skinner-wan} (and  recalled in Section~\ref{sec:Sel} below) which differs from $\mathrm{Sel}_{\mathrm{Gr}}(K,\mathbf{A})$ by the defining local conditions at the primes dividing $p$ (and possibly $N^-$):
\begin{conj}\label{conj:BDP}
	The Pontrjagin dual $X^{N^-}_{\emptyset,0}$ of $\mathrm{Sel}^{N^-}_{\emptyset,0}(K,\mathbf{A})$ is $\Lambda$-torsion, and we have
	\[
	\mathrm{char}(X^{N^-}_{\emptyset,0})\Lambda^{\mathrm{ur}}=(\mathscr{L}_{\mathfrak{p}}^{BDP})^2.
	\]
\end{conj}

As a key step in our proof, in Section~\ref{sec:equiv-IMC} we establish the \emph{equivalence} between Conjecture~\ref{conj:HPMC} and Conjecture~\ref{conj:BDP}. In particular, we show that Howard's divisibility implies the divisibility ``$\supseteq$'' in Conjecture~\ref{conj:BDP}. In Section~\ref{sec:equiv-spval}, assuming
\begin{equation}\label{eq:alg-rank1}
\mathrm{rank}_{\mathbf{Z}}E(K)=1,
\end{equation}
by a useful commutative algebra lemma from \cite{skinner-urban} and the ``anticyclotomic control theorem'' of \cite{jetchev-skinner-wan},  we reduce the proof of the opposite divisibility 
%(and hence of Conjecture~\ref{conj:HPMC}) 
to the proof of the equality
\begin{equation}\label{eq:exact-bound}
[E(K):\mathbf{Z}.P]^2=\#\sha(E/K)[p^\infty]\prod_{\ell\vert N^+}c_\ell^2
\end{equation}
up to a $p$-adic unit, where 
%is a $p$-unit multiple of a Heegner point $y_{K} \in E(K)$ 
%arising from the modular parametrisation 
$P\in E(K)$ is a $p$-primitive generator 
of $E(K)$ up to torsion,  
and $c_\ell$ is the Tamagawa number of $E/\mathbf{Q}_\ell$. 
%By taking the $N^-$-minimal local condition, we miss the Tamagawa numbers at primes dividing $N^-$; however, under our assumption, those Tamagawa numbers become $p$-adic units. 
Under the hypotheses of Theorem~\ref{thm:main}, equalities $(\ref{eq:alg-rank1})$ and $(\ref{eq:exact-bound})$ follow from the Gross--Zagier formula  \cite{gross-zagier-original, yuan-zhang-zhang}, and the work of Kolyvagin \cite{kolyvagin-selmer}, and W.~Zhang \cite{wei-zhang-mazur-tate}, with $P\in E(K)$  given by the trace of a Heegner point defined over the Hilbert class field of $K$, yielding our main result. 

%In order to get the equality, we consider the following implications.
%\[
%\xymatrix{
%	\mathrm{char}\left( \dfrac{ \mathrm{H}^1_{\mathcal{F}_\Lambda} \left( K, T \otimes \Lambda \right) }{ \Lambda \kappa^\infty_1 }\right) \subseteq \mathrm{char}(M_{\infty}) \ar@{<=>}[r]^-{\textrm{ \cite{wan-heegner} }} & (L^{\mathrm{BDP}}_p) \subseteq \mathrm{char}_{\Lambda}(X_{\wp}) \ar@/^/@{=>}[d]^-{\mathrm{control}}\\
%	\mathrm{length}_{\mathbf{Z}_p}(\sha(E/K)[p^\infty]) \leq 2 \cdot \mathrm{length}_{\mathbf{Z}_p}\left( [E(K):\mathbf{Z}y_K]\right) \ar@{<=>}[r]^-{\textrm{\cite[(3.5.d)]{jetchev-skinner-wan}}} & 
%	\mathrm{length}_{\mathbf{Z}_p}( S_{ac} ) \leq 2 \cdot \mathrm{length}_{\mathbf{Z}_p}\left( L^{\mathrm{BDP}}_p (0)\right)  \ar@/^/@{=>}[u]^-{\textrm{ \cite[Lemma 3.2]{skinner-urban} for equality}}
%}
%\]
%Using the indivisibility of Heegner points, we can show the structure theorem for Shafarevich-Tate groups (McCallum's survey on Kolyvagin's Annalen).
%This implies
%$$\mathrm{length}_{\mathbf{Z}_p}(\sha(E/K)[p^\infty]) = 2 \cdot \mathrm{length}_{\mathbf{Z}_p}\left( [E(K):\mathbf{Z}y_K]\right).$$
%Using the above diagram (again), we get the equality.

%
%Many inspirations come from Kurihara's papers \cite{kurihara-iwasawa-2012} and \cite{kurihara-munster}.
%This article is an anticyclotomic cousin of \cite{kks} with input from \cite{wei-zhang-mazur-tate}.

\begin{rem}\label{rem:higher-rank}
By the work of Cornut--Vatsal \cite{cornut-vatsal}, the Heegner points $y_n\in E(K_n)$ defined over the $n$-th layer of the anticyclotomic $\mathbf{Z}_p$-extension are non-torsion for $n$ sufficiently large. Taking one such $n$, and letting 
\[
y_{n,\chi}\in E(K_n)^\chi\subset E(K_n)\otimes_{\mathbf{Z}[\mathrm{Gal}(K_n/K)]}\mathbf{Z}[\chi]
\] 
be the image of $y_n$ in the $\chi$-isotypical component of for a primitive character $\chi:\mathrm{Gal}(K_n/K)\rightarrow\mathbf{Z}[\chi]^\times$, one can use Kolyvagin's methods (as extended in \cite{bertolini-darmon-crelles}) 
%to ring class fields) 
to establish the rank one property of $E(K_n)^\chi$, and the Gross--Zagier formula \cite{yuan-zhang-zhang} combined with a generalization of Kolyvagin's structure theorem  for Shafarevich--Tate groups \cite{kolyvagin-structure-sha} should yield an analogue of $(\ref{eq:exact-bound})$ in terms of the index of $y_{n,\chi}$.
%in $E(K_n)^\chi$. 
% As we intend to pursue in forthcoming work, with these results in hand the analytic rank $1$ hypothesis can be removed from our Theorem~\ref{thm:main}.
\end{rem}

%\subsection{Acknowledgements} 

\section{Selmer structures}\label{sec:Sel}

We keep the notations from the Introduction. In particular, $E/\mathbf{Q}$ is an elliptic curve of conductor $N$ with good ordinary reduction at a prime $p>3$, and $K$ is an imaginary quadratic field of discriminant prime to $Np$ in which $p=\mathfrak{p}\overline{\mathfrak{p}}$ splits. Throughout the rest of this paper, we also fix once and for all a choice of complex and $p$-adic embeddings  $\mathbf{C}\overset{\iota_\infty}\hookleftarrow\overline{\mathbf{Q}}\overset{\iota_p}\hookrightarrow\overline{\mathbf{Q}}_p$.

Let $\Sigma$ be a finite set of places of $K$ including the places lying above $p$, $\infty$ and the primes dividing $N$. 
For a finite extension $F$ of $K$, let $F_{\Sigma}$ denote the maximal extension of $F$ unramified outside the places lying above $\Sigma$.
Following \cite{mazur-rubin-book}, given a Selmer structure $\mathcal{F} = \lbrace \mathcal{F}_w \rbrace_{w \mid v, v \in \Sigma}$ on a $G_K$-module $M$, we define the associated Selmer group $\mathrm{Sel}_{\mathcal{F}} (F, M)$ by
\[
\mathrm{Sel}_{\mathcal{F}} (F, M) := \mathrm{ker} \biggl( \mathrm{H}^1(F_{\Sigma}/F, M) \longrightarrow \prod_{w} \dfrac{\mathrm{H}^1(F_w, M)}{\mathrm{H}^1_{\mathcal{F}}(F_w, M)} \biggr).
\]

 If $M$ is a $G_K$-module and $L/K$ a finite Galois extension, we have the induced representation 
\[
\mathrm{Ind}_{L/K}M := \lbrace f : G_K \to M : f(\sigma x) = f(x)^{\sigma} \textrm{ for all } x \in G_K, \sigma \in G_L \rbrace,
\]
which is equipped with commuting actions of $G_K$ and $\mathrm{Gal}(L/K)$. Consider the modules   
\[
{\displaystyle \mathbf{T} := \varprojlim \left( \mathrm{Ind}_{K_n/K}T\right)}, %\simeq T \otimes \Lambda$ with corestriction and
\quad
{\displaystyle \mathbf{A} := \varinjlim \left( \mathrm{Ind}_{K_n/K} A\right)} \simeq \mathrm{Hom}(\mathbf{T}, \mu_{p^\infty}), 
\]
where the limits are with respect to the corestriction and restriction maps, respectively. These are finitely and cofinitely generated over $\Lambda$, respectively. 
 
We recall the ordinary filtrations at $p$. Let $G_{\mathbf{Q}_p}:=\mathrm{Gal}(\overline{\mathbf{Q}}_p/\mathbf{Q}_p)$, viewed as a decomposition group at $p$ inside $G_{\mathbf{Q}}$ via $\iota_p$. By $p$-ordinarity, there is a one-dimensional $G_{\mathbf{Q}_p}$-stable subspace $\mathrm{Fil}^+V\subset V$ such that the $G_{\mathbf{Q}_p}$-action on the quotient $\mathrm{Fil}^-V := V/\mathrm{Fil}^+V$ is unramified. Set
\[
\mathrm{Fil}^+T := T \cap \mathrm{Fil}^+V,\quad
\mathrm{Fil}^-T := T / \mathrm{Fil}^+T,\quad 
\mathrm{Fil}^+A :=  \mathrm{Fil}^+V / \mathrm{Fil}^+T,\quad \mathrm{Fil}^-A := A/\mathrm{Fil}^+A,
\]
and define the submodules $\mathrm{Fil}^+\mathbf{T}\subset\mathbf{T}$ and $\mathrm{Fil}^+\mathbf{A}\subset\mathbf{A}$ by
\[ 
\mathrm{Fil}^+\mathbf{T} := \varprojlim\mathrm{Ind}_{K_n/K}\mathrm{Fil}^+T,\quad 
\mathrm{Fil}^+\mathbf{A} := \varinjlim\mathrm{Ind}_{K_n/K}\mathrm{Fil}^+A,
\]
and set $\mathrm{Fil}^-\mathbf{T}:=\mathbf{T}/\mathrm{Fil}^+\mathbf{T}$ and $\mathrm{Fil}^-\mathbf{A}:=\mathbf{A}/\mathrm{Fil}^+\mathbf{A}$. 

Following the terminology introduced in \cite[$\S$2]{castella-beilinson-flach}, if $M$ denotes any of the $G_K$-modules above, we consider the following three local conditions at a place $v$ lying above $p$:
\begin{align*}
\mathrm{H}^1_{\emptyset}(F_v, M) & := \mathrm{H}^1(F_v, M), \\
\mathrm{H}^1_{\mathrm{Gr}}(F_v, M) & := \ker\left(\mathrm{H}^1(F_v, M) \longrightarrow \mathrm{H}^1(F_v, \mathrm{Fil}^- M )\right), \\
\mathrm{H}^1_{0}(F_v, M) & := 0.
\end{align*}
We also recall two local conditions at a place $v$ not lying above $p$:
\[
\xymatrix{
\mathrm{H}^1_{\mathrm{triv}}(F_v, M)  :=0, &
\mathrm{H}^1_{\mathrm{ur}}(F_v, M)  := \mathrm{H}^1(F_v/I_v,  M^{I_v} ) . 
}
\]
If $F = K$ and $M = \mathbf{A}$, then 
$\mathrm{H}^1_{\mathrm{triv}}(F_v, M) = \mathrm{H}^1_{\mathrm{ur}}(F_v, M)$ unless $v$ divides $N^-$ and $\overline{\rho}$ is unramified at $v$ as in \cite[Page 1362]{pw-mu}.

Using these local conditions, for $a, b \in \lbrace \emptyset, \mathrm{Gr}, 0 \rbrace$ we define
\[
\mathrm{Sel}_{a,b}(K, \mathbf{T}) 
:= \mathrm{ker} \biggr( 
\mathrm{H}^1(K_{\Sigma}/K, \mathbf{T}) \to \dfrac{\mathrm{H}^1(K_{\mathfrak{p}}, \mathbf{T})}{\mathrm{H}^1_{a}(K_{\mathfrak{p}}, \mathbf{T})} \times \dfrac{\mathrm{H}^1(K_{\overline{\mathfrak{p}}}, \mathbf{T})}{\mathrm{H}^1_{b}(K_{\overline{\mathfrak{p}}}, \mathbf{T})} \times \prod_{v \in \Sigma, v \nmid p} \dfrac{\mathrm{H}^1(K_v, \mathbf{T})}{\mathrm{H}^1_{\mathrm{ur}}(K_v, \mathbf{T})} \biggl).
\]
In particular, $\mathrm{Sel}_{\mathrm{Gr}}(K, \mathbf{T}):= \mathrm{Sel}_{\mathrm{Gr},\mathrm{Gr}}(K, \mathbf{T})$ coincides with the $\Lambda$-adic Selmer group of \cite[Def.~2.2.6]{howard-kolyvagin}, which we shall denote by $\mathrm{H}^1_{\mathcal{F}_\Lambda}(K,\mathbf{T})$ following the notation in \cite{howard-kolyvagin}. The same definitions and notational convention applies to $\mathbf{A}$.
Putting the trivial local condition at primes dividing $N^-$, 
we can also define the $N^-$-minimal variant of discrete Selmer group by
\[
\mathrm{Sel}^{N^-}_{a,b}(K, \mathbf{A}) 
:= \mathrm{ker} \biggr( \mathrm{Sel}_{a,b}(K, \mathbf{A}) \to  \prod_{v \mid N^-} \mathrm{H}^1_{\mathrm{ur}}(K_v, \mathbf{A})   \biggl) .
\]
\begin{rem}
If $v$ divides $N^-$ and $\overline{\rho}$ is unramified at $v$, then
$\mathrm{H}^1_{\mathrm{triv}}(K_v, \mathbf{A}) \neq \mathrm{H}^1_{\mathrm{ur}}(K_v, \mathbf{A})$ since $v$ splits completely in $K_{\infty}/K$. See \cite[Lem.~3.4]{pw-mu} for the exact difference.
Indeed, the $N^-$-minimal Selmer groups are practically preferred in the anticyclotomic Iwasawa theory for modular forms since 
the mod $p^n$ Selmer groups with the $N^-$-ordinary local condition \cite[$\S$3.2]{pw-mu}, which is used in the Euler system argument \`{a} la Bertolini-Darmon \cite{bertolini-darmon-imc-2005}, becomes the minimal Selmer group after taking the limit with respect to $n$ under the tame ramification condition described in Remark \ref{rem:conditionCR}. See \cite[Prop.~3.6]{pw-mu} for detail.
\end{rem}

%Note that the Greenberg ordinary local condition coincides with the Bloch--Kato local condition \cite[Lemma 2]{flach-cassels-tate} since $H^1_f = H^1_g$ in our setup due to the Weil conjecture.

%Following \cite[Def.~2.2.6]{howard-kolyvagin}, we define the $\Lambda$-adic ordinary Selmer group by
%$$\mathrm{Sel}_{\mathcal{F}_{\Lambda}}(K, \mathbf{T}) 
%:= \mathrm{ker} \left( 
%\mathrm{H}^1(K_{\Sigma}/K, \mathbf{T}) \to \prod_{v \in \Sigma} \dfrac{\mathrm{H}^1(K_v, \mathbf{T})}{\mathrm{H}^1_{\mathcal{F}_{\Lambda}}(K_v, \mathbf{T})} \right)$$
%where
%$$\mathrm{H}^1_{\mathcal{F}_{\Lambda}}(K_v, \mathbf{T}) = 
%\left \lbrace
%    \begin{array}{ll}
%       \mathrm{H}^1_{\mathrm{ur}}(K_v, \mathbf{T})  & \textrm{if} \ v \nmid p \\ 
%      \mathrm{im} \left( \mathrm{H}^1(K_v, \mathrm{Fil}\mathbf{T}) \to \mathrm{H}^1(K_v, \mathbf{T}) \right) & \textrm{if} \ v \mid p
%    \end{array}
%    \right .$$
%The same definition applies to $\mathbf{A}$.

%Note that the locally trivial condition and the locally unramified local condition at primes not dividing $p$ coincide under the assumption $\overline{\rho}$ is ramified at all primes dividing $N^-$.

%\section{Indivisibility of Heegner points and the structure of Shafarevich--Tate groups}
\section{Heegner point Kolyvagin systems}

In this section we briefly recall the statement of Howard's theorems towards Conjecture~\ref{conj:HPMC}, as well as the results from Wei~Zhang's proof of Kolvyagin's conjecture %\cite{wei-zhang-mazur-tate} 
that we shall use to upgrade Howard's divisibility to an equality.

%\section{Heegner point Kolyvagin systems}
%\subsection{Howard's theorem} 

Let 
\[
X := \mathrm{Hom}_{\mathbf{Z}_p}\left(\mathrm{Sel}_{\mathrm{Gr}}(K, \mathbf{A}) , \mathbf{Q}_p/\mathbf{Z}_p\right) 
\]
be the Pontrjagin dual of the $\Lambda$-adic Greenberg Selmer group. Since we shall not directly need it here, we refer the reader to \cite[\S{1.2}]{howard-kolyvagin} for the definition of a Kolyvagin system $\boldsymbol{\kappa}^\infty=\{\kappa_n^\infty\}_{n\in\mathcal{N}}$ (attached to a $G_K$-module $M$ together with a Selmer structure $\mathcal{F}$), where $n$ runs over the set of square-free products of certain primes inert in $K$, with the convention that $1\in\mathcal{N}$. 

\begin{thm}%[Howard, Cornut--Vatsal]
\label{thm:howard}
Assume that $p>3$ is a good ordinary prime for $E$, $D_K$ is coprime to $pN$, and $\overline{\rho}$ is surjective. Let $\mathcal{F}_\Lambda$ be the Selmer structure for the Greenberg Selmer group. Then: 
%$\mathrm{H}^1_{\mathcal{F}_\Lambda}(K,\mathbf{T})$ and  $\mathrm{H}^1_{\mathcal{F}_\Lambda}(K,\mathbf{A})$ both have rank $1$, and there is a torsion $\Lambda$-module $M_\infty$ with: 
%\begin{enumerate}
%	\item $\mathrm{char}(M_\infty)=\mathrm{char}(M_\infty)^\iota$,
%	\item $X\sim\Lambda\oplus M_\infty\oplus M_\infty$,
%	\item $\mathrm{char}(M_{\infty})=\mathrm{char}\biggl( \dfrac{ \mathrm{H}^1_{\mathcal{F}_\Lambda} \left( K, \mathbf{T} \right) }{ \Lambda\kappa^\infty_1 }\biggr)$.	
%\end{enumerate}
\begin{enumerate}
\item There exists %a $\Lambda$-adic Heegner point 
a $\Lambda$-adic Kolyvagin system $\ks^\infty$ for $(\mathbf{T},\mathcal{F}_\Lambda)$ with $\kappa^{\infty}_1 \neq 0$,
\item $\mathrm{Sel}_{\mathrm{Gr}}( K, \mathbf{T} )$ is a torsion-free, rank one $\Lambda$-module.
\item There is a torsion $\Lambda$-module $M_{\infty}$ such that $\mathrm{char}(M_{\infty}) = \mathrm{char}(M_{\infty})^{\iota}$ and a pseudo-isomorphism
\[
X \sim \Lambda \oplus M_{\infty} \oplus M_{\infty}.
\]
\item $\mathrm{char}(M_{\infty})$ divides $\mathrm{char}\bigl( \mathrm{Sel}_{\mathrm{Gr}}(K,\mathbf{T})/\Lambda \kappa^\infty_1)$. 
%In other words,
%$$ \mathrm{char}\left( \dfrac{ \mathrm{H}^1_{\mathcal{F}_\Lambda} \left( K, T \otimes \Lambda \right) }{ \Lambda \kappa^\infty_1 }\right) \subseteq \mathrm{char}(M_{\infty}).$$
\end{enumerate}
\end{thm}

\begin{proof}
This is \cite[Thm.~2.2.10]{howard-kolyvagin}, as extended in \cite[Thm.~3.4.2]{howard-gl2-type} to the case $N^-\neq 1$. The non-triviality of $\kappa_1^\infty$ follows from the work of Cornut--Vatsal \cite{cornut-vatsal}.
\end{proof}

%\begin{rem}
%The second statement nearly corresponds to the core rank one property for Kolyvagin systems \`{a} la Mazur-Rubin \cite[Theorem 5.2.14 and Lemma 5.3.5]{mazur-rubin-book}.
%\end{rem}

%\subsection{Wei Zhang's theorem}

Following \cite{wei-zhang-mazur-tate}, we say that a prime $\ell$ is called a \emph{Kolyvagin prime} if $\ell$ is prime to $pND_K$, inert in $K$, and the index
\[
M(\ell):=\mathrm{min}_p\{v_p(\ell+1),v_p(a_\ell)\}
\] 
is strictly positive, where $a_\ell=\ell+1-\#E(\mathbf{F}_\ell)$. Let 
\[
\delta_w:E(F_w)\otimes_{\mathbf{Z}}\mathbf{Z}_p\longrightarrow \mathrm{H}^1(F_w,T)
\]
be the local Kummer map, and let $\mathcal{F}$ be the Selmer structure on $T$ given by $\mathrm{H}^1_{\mathcal{F}}(F_w,T):=\mathrm{im}(\delta_w)$. As explained in \cite[\S{1.7}]{howard-kolyvagin} (and its extension in \cite[\S{2.3}]{howard-gl2-type} to $N^-\neq 1$), Heegner points give rise to a (mod $p^M$) Kolyvagin system
\begin{equation}\label{eq:k-infty}
\ks =\left\{ \kappa_n = c_M(n)\in \mathrm{H}^1(K,E[p^M]):0<M\leqslant M(n),\;  n\in\mathcal{N}\right\}\nonumber
\end{equation}
for $(T/p^M T,\mathcal{F})$, where $\mathcal{N}$ denotes the set of square-free products of Kolyvagin primes, and for $n\in\mathcal{N}$ we set $M(n):=\min\{M(\ell):\ell\vert n\}$, with $M(1)=\infty$ by convention.

\begin{thm}%[Wei~Zhang] 
\label{thm:indivisibility}
Assume that: 
\begin{itemize}
	\item $p>3$ is a good ordinary prime for $E$,
	\item $D_K$ is coprime to $pN$,
	\item Condition~$\heartsuit$ holds for $(E,p,K)$,
	\item $G_K\rightarrow\mathrm{Aut}_{\mathbf{F}_p}(E[p])$ is surjective.
\end{itemize}
Then the collection of mod $p$ cohomology classes
\begin{equation}\label{eq:mod-p}
\ks = \{\kappa_n = c_1(n)\in \mathrm{H}^1(K,E[p]):n\in\mathcal{N}\}
\end{equation}
is nonzero. In particular, $\kappa_n \neq 0$ for some $n$.  %Kolyvagin's conjecture holds.
\end{thm}

\begin{proof}
This is \cite[Thm.~9.3]{wei-zhang-mazur-tate}.
\end{proof}

\begin{rem}
In the terminology of %\cite[Def.~4.5.5]
\cite{mazur-rubin-book}, Wei Zhang's Theorem~\ref{thm:indivisibility} may be interpreted as establishing the \emph{primitivity}  
%(under the hypotheses of Theorem~\ref{thm:indivisibility}) 
of the system $\ks$.  
Mazur--Rubin also introduced the (weaker) notion of $\Lambda$-\emph{primitivity} for the cyclotomic analogue of $\ks^{\infty}$ (see \cite[Def.~5.3.9]{mazur-rubin-book}), and in some sense our main result in this paper may be seen as a realization of the implications
\[
\textrm{$\ks$ is primitive}\;\Longrightarrow\;
\textrm{$\boldsymbol{\kappa}^\infty$ is $\Lambda$-primitive}\;\Longrightarrow\;
\textrm{Conjecture~\ref{conj:HPMC} holds},
\]
where $\boldsymbol{\kappa}^\infty$ is Howard's Heegner point $\Lambda$-adic Kolyvagin system from Theorem~\ref{thm:howard}. 
\end{rem}

Combined with Kolyvagin's work, Theorem~\ref{thm:indivisibility} yields the following exact formula the order of $\#\sha(E/K)[p^\infty]$ that we shall need.

\begin{cor}%[Kolyvagin] 
	\label{cor:structure_of_sha}
Let the hypotheses be as in Theorem \ref{thm:indivisibility}. If $\mathrm{ord}_{s=1}L(E/K,s)=1$, then
\[
\mathrm{ord}_p(\#\sha(E/K)[p^\infty]) = 2 \cdot \mathrm{ord}_p [E(K):\mathbf{Z}.y_K]
\]
where $y_K\in E(K)$ is a Heegner point.
\end{cor}

\begin{proof} 
After Theorem~\ref{thm:indivisibility} (more precisely, the non-vanishing of (\ref{eq:mod-p})), this follows from Kolyvagin's structure theorem from $\sha(E/K)$ \cite{kolyvagin-structure-sha} 
(see also \cite{mccallum-kolyvagin}), 
using that $y_K$ has infinite order by the Gross--Zagier formula \cite{gross-zagier-original, yuan-zhang-zhang} (\emph{cf.} \cite[Thm.~10.2]{wei-zhang-mazur-tate}).	
\end{proof}

\section{Equivalent main conjectures}\label{sec:equiv-IMC}

In this section we establish the equivalence between Conjecture~\ref{conj:HPMC} (the Heegner point main conjecture) and Conjecture~\ref{conj:BDP} (the Iwawawa--Greenberg main conjecture for $\mathscr{L}_{\mathfrak{p}}^{BDP}$) in the Introduction. 

To ease the notation, for $a, b\in\{\emptyset,\mathrm{Gr},0\}$ we let $X_{a,b}$ denote the Pontrjagin dual of the generalized Selmer group $\mathrm{Sel}_{a,b}(K,\mathbf{A})$, keeping the earlier convention that $X:=X_{\mathrm{Gr},\mathrm{Gr}}$. %and $X_{\mathfrak{p}}:=X_{\emptyset,0}$.

\begin{thm}\label{thm:equiv-imc}
	Suppose $E[p]$ is ramified at all primes $\ell\vert N^-$. Then 
	Conjectures~\ref{conj:HPMC} and \ref{conj:BDP} are equivalent. More precisely, $X$ has $\Lambda$-rank $1$ if and only $X_{\emptyset,0}$ is $\Lambda$-torsion, and one-sided divisibility holds in Conjecture~\ref{conj:HPMC}(3) if and only if the same divisibility holds in Conjecture~\ref{conj:BDP}.  
\end{thm}

\begin{proof}
	This is essentially shown in the Appendix of \cite{castella-beilinson-flach} (\emph{cf.} \cite[\S{3.3}]{wan-heegner}); all the references in the proof that follows are to results in that paper.
     First, note that $\mathrm{Sel}_{a,b}(K,\mathbf{A}) = \mathrm{Sel}^{N^-}_{a,b}(K,\mathbf{A})$ by assumption.	
	 If $X$ has $\Lambda$-rank $1$, then $\mathrm{Sel}_{\mathrm{Gr}}(K,\mathbf{T})$ has $\Lambda$-rank $1$ by Lemma~2.3(1), and hence  $X_{\emptyset,0}$ is $\Lambda$-torsion by Lemma~A.4. Conversely, assume that $X_{\emptyset,0}$ is $\Lambda$-torsion. Then $X_{\mathrm{Gr},0}$ is also $\Lambda$-torsion (see eq.~(A.7)), and so $X_{\mathrm{Gr},\emptyset}$ has $\Lambda$-rank $1$ by Lemma~2.3(2). Now, global duality yields the exact sequence
	\begin{equation}\label{eq:PT}
	0\longrightarrow\mathrm{coker}(\mathrm{loc}_{\mathfrak{p}})\longrightarrow
	X_{\emptyset,\mathrm{Gr}}\longrightarrow X\longrightarrow 0,
	\end{equation}
	where $\mathrm{loc}_{\mathfrak{p}}:\mathrm{Sel}_{\mathrm{Gr}}(K,\mathbf{T})\rightarrow \mathrm{H}^1_{\mathrm{Gr}}(K_{\mathfrak{p}},\mathbf{T})$ is the restriction map. Since $\mathrm{H}^1_{\mathrm{Gr}}(K_{\mathfrak{p}},\mathbf{T})$ has $\Lambda$-rank $1$, the leftmost term in $(\ref{eq:PT})$ is $\Lambda$-torsion by Theorem~A.1 and the nonvanishing of $\mathscr{L}_\mathfrak{p}^{BDP}$ (see Theorem~1.5); since $X_{\mathrm{Gr},\emptyset}\simeq X_{\emptyset,\mathrm{Gr}}$ by the action of complex conjugation, we  conclude from $(\ref{eq:PT})$ that $X$ has $\Lambda$-rank $1$. 
	
	Next, assume that $X$ has $\Lambda$-rank $1$. By Lemma~2.3(1), this amounts to the assumption that $\mathrm{Sel}_{\mathrm{Gr}}(K,\mathbf{T})$ has $\Lambda$-rank $1$, and so by Lemmas~A.3 and A.4 for every height one prime $\mathfrak{P}$ of $\Lambda$ we have
	\begin{equation}\label{eq:A3}
	\mathrm{length}_{\mathfrak{P}}(X_{\emptyset,0})=\mathrm{length}_{\mathfrak{P}}(X_{\mathrm{tors}})+2\;\mathrm{length}_{\mathfrak{P}}(\mathrm{coker}(\mathrm{loc}_{\mathfrak{p}})),
	\end{equation} 
	where $X_{\mathrm{tors}}$ denotes the $\Lambda$-torsion submodule of $X$, 
	and for every height one prime $\mathfrak{P}'$ of $\Lambda^{\mathrm{ur}}$ 
	\begin{equation}\label{eq:A4}
	\mathrm{ord}_{\mathfrak{P}'}(\mathscr{L}_{\mathfrak{p}}^{BDP})=\mathrm{length}_{\mathfrak{P}'}(\mathrm{coker}(\mathrm{loc}_{\mathfrak{p}})\Lambda^{\mathrm{ur}})+\mathrm{length}_{\mathfrak{P}'}\biggl(\dfrac{ \mathrm{Sel}_{\mathrm{Gr}}(K, \mathbf{T})\Lambda^{\mathrm{ur}} }{ \Lambda^{\mathrm{ur}} \kappa^\infty_1 }\biggr).
	\end{equation}
	Thus for any height one prime $\mathfrak{P}$ of $\Lambda$, letting $\mathfrak{P}'$ denote its extension to $\Lambda^{\mathrm{ur}}$, we see from $(\ref{eq:A3})$ and $(\ref{eq:A4})$ that
	\[
	\mathrm{length}_{\mathfrak{P}}(X_{\mathrm{tors}})\leqslant 2\;\mathrm{length}_{\mathfrak{P}}\biggl(\dfrac{ \mathrm{Sel}_{\mathrm{Gr}}(K, \mathbf{T}) }{ \Lambda \kappa^\infty_1 }\biggr)\quad\Longleftrightarrow\quad
	\mathrm{length}_{\mathfrak{P}}(X_{\emptyset,0})\leqslant 2\;\mathrm{ord}_{\mathfrak{P}'}(\mathscr{L}_{\mathfrak{p}}^{BDP}),
	\]
	and similarly for the opposite inequalities. The result follows. 
\end{proof}

\begin{rem}
Accounting for the difference between the unramified (as implicitly used here) and the strict local conditions in $\mathrm{H}^1(F_w,\mathbf{A})$ for $w\vert\ell\vert N^-$ in terms of $p$-parts of the corresponding
Tamagawa numbers (see e.g. \cite[\S{3}]{pw-mu}), it is possible to prove an analogue of Theorem~\ref{thm:equiv-imc} without the above ramification hypothesis on $E[p]$. Indeed, the difference will only affect the $\mu$-invariants of both sides due to \cite[Lem. 3.4]{pw-mu}.
%\footnote{In order to say it is the exact difference, we need to know the global-to-local map defining $\mathrm{Sel}_{\emptyset, 0}$ is surjective. See \cite[Prop 5.1]{pw-mu}. Also, we need to specify the Selmer structure at primes not dividing $p$ to see the difference more clearly in Section 2.}
\end{rem}

\section{Equivalent special value formulas}\label{sec:equiv-spval}

The goal of this section is to establish Corollary~\ref{cor:equiv-spval} below, which is a manifestation of the equivalence of Theorem~\ref{thm:equiv-imc} after specialization at the trivial character.

\begin{thm}\label{thm:equiv-spval}
	Assume that $\mathrm{rank}_{\mathbf{Z}}E(K)=1$, $\#\sha(E/K)<\infty$, and $E[p]$ is irreducible as $G_{\mathbf{Q}}$-module. Then $X^{N^-}_{\emptyset, 0}$ is $\Lambda$-torsion, and letting $f_{\emptyset,0}(T)\in\mathbf{Z}_p\llbracket T \rrbracket$ be a generator of its characteristic ideal, the following equivalence holds:
	\[
	f_{\emptyset,0}(0)\sim_p\biggl(\dfrac{1-a_p+p}{p}\biggr)^2\cdot\log_{\omega_E}(P)^2\quad\Longleftrightarrow\quad[E(K):\mathbf{Z}.P]^2\sim_p\#\sha(E/K)[p^\infty]\prod_{\ell\mid N^+}c_\ell^2,
	\]
	where $P\in E(K)$ is any generator of $E(K)\otimes_{\mathbf{Z}}\mathbf{Q}$, $c_\ell$ is the Tamagawa number of $E/\mathbf{Q}_\ell$,  and $\sim_p$ denotes equality up to a $p$-adic unit.
\end{thm}
\begin{rem}
Note that no Tamagawa defect at the primes dividing $N^-$ is assigned in the RHS due to the $N^-$-minimal local condition of the Selmer group. Indeed, $c_\ell$ for $\ell$ dividing $N^-$ becomes trivial in our setting due to Condition $\heartsuit$ (Condition \ref{assu:heart}.(2)). See \cite[Prop. 3.7]{pw-mu} for the definite case.
\end{rem}
\begin{proof}
	As shown in \cite[p.~395-6]{jetchev-skinner-wan}, our assumptions imply hypotheses (corank~1), (sur), and (irred$_{\mathcal{K}}$) of \cite[\S{3.1}]{jetchev-skinner-wan}, and so by [\emph{loc.~cit.}, Thm.~3.3.1] (with $S=S_p$ the set of primes dividing $N$ and $\Sigma=\emptyset$) the module $X^{N^-}_{\emptyset, 0}$ is $\Lambda$-torsion, and  
	\begin{equation}\label{eq:control}
	\# \mathbf{Z}_p / f_{\emptyset, 0}(0) = \# \mathrm{Sel}^{N^-}_{\emptyset, 0}(K, E[p^\infty]) \cdot C^{}(E[p^\infty]),
	\end{equation}
	where
	\[
	C^{}(E[p^\infty]) := \# \mathrm{H}^0(K_{\mathfrak{p}}, E[p^\infty]) \cdot \# \mathrm{H}^0(K_{\overline{\mathfrak{p}}}, E[p^\infty]) \cdot \prod_{w\mid N^+} \# \mathrm{H}^1_{\mathrm{ur}}(K_w, E[p^\infty]).
	\]
	On the other hand, from \cite[(3.5.d)]{jetchev-skinner-wan} we have
	\begin{equation}\label{eq:ev-index}
	\# \mathrm{Sel}^{N^-}_{\emptyset, 0} (K, E[p^\infty]) 
	=
	\# \sha(E/K)[p^\infty] \cdot \biggl(  \dfrac{ \# (\mathbf{Z}_p / ( \frac{1-a_p+p}{p}) \cdot \mathrm{log}_{\omega_E} P) }{ [E(K): \mathbf{Z}.P]_p \cdot \# \mathrm{H}^0(K_{\mathfrak{p}}, E[p^\infty]) } \biggr)^2,
	\end{equation}
	where $P\in E(K)$ is any generator of $E(K)\otimes_{\mathbf{Z}}\mathbf{Q}$, and $[E(K): \mathbf{Z}.P]_p$ 
	denotes the $p$-part of the index $[E(K): \mathbf{Z}.P]$. Combining $(\ref{eq:control})$ and $(\ref{eq:ev-index})$ we thus arrive at
	\[
	\# \mathbf{Z}_p / f_{\emptyset, 0}(0) = 
	\# \sha(E/K)[p^\infty]\cdot\biggl(  \dfrac{ \# (\mathbf{Z}_p / ( \frac{1-a_p+p}{p}) \cdot \mathrm{log}_{\omega_E} P) }{ [E(K): \mathbf{Z}.P]_p }\biggr)^2\cdot\prod_{w\mid N^+} \# \mathrm{H}^1_{\mathrm{ur}}(K_w, E[p^\infty]).
	\]
	Since the order of $\mathrm{H}^1_{\mathrm{ur}}(K_w,E[p^\infty])$ is the $p$-part of the Tamagawa number of $E/K_w$, the result follows.
\end{proof}

The fundamental $p$-adic Waldspurger formula due to Bertolini--Darmon--Prasanna \cite{bertolini-darmon-prasanna-duke} will allow us to relate the left-hand side of Theorem~\ref{thm:equiv-spval} to the anticyclotomic main conjecture.

\begin{thm}[Bertolini--Darmon--Prasanna]\label{thm:bdp}
The following special value formula holds:
\[
\mathscr{L}^{BDP}_{\mathfrak{p}}(0) = \biggl( \frac{1-a_p+p}{p} \biggr) \cdot \bigl( \mathrm{log}_{\omega_E} y_K \bigr),
\]
where $y_K\in E(K)$ is a Heegner point.
\end{thm}

\begin{proof}
This is a special case of \cite[Theorem 5.13]{bertolini-darmon-prasanna-duke} (\emph{cf.} \cite[Theorem 3.12]{bertolini-darmon-prasanna-pacific}) and \cite[Theorem 1.1]{brooks}.
%See also \cite[Proposition 5.1.6 and 5.1.7]{jetchev-skinner-wan}.
\end{proof}

\begin{cor}\label{cor:equiv-spval}
With notations and hypotheses as in Theorem~\ref{thm:equiv-spval}, assume in addition that $(E,p,K)$ satisfies Condition~$\heartsuit$. Then the following equivalence holds:
\[
f_{\emptyset,0}(0)\sim_p \mathscr{L}_{\mathfrak{p}}^{BDP}(0)^2\quad\Longleftrightarrow\quad[E(K):\mathbf{Z}.P]^2\sim_p\#\sha(E/K)[p^\infty]
\]
for $P$ a $p$-unit multiple of the Heegner point $y_{K} \in E(K)$.
\end{cor}

\begin{proof}
Since Condition~$\heartsuit$ forces all the Tamagawa numbers $c_\ell$ for the primes $\ell\vert N^+$ to be $p$-adic units, the result follows	from Theorem~\ref{thm:equiv-spval} and Theorem~\ref{thm:bdp}.
\end{proof}
\begin{rem} 
Here we require $P$ to be a $p$-unit multiple of the Heegner point $y_{K}$, as otherwise the logarithm $\log_{\omega_{E}}P$ and the index $[E(K):\mathbb{Z}.P]$ can be divisible by an extra power of $p$. 
\end{rem}

\section{Skinner--Urban lifting lemma}

We recall the following ``easy lemma" in \cite{skinner-urban}.

\begin{lem}%[{\hspace{1sp}\cite[Lemma 3.2]{skinner-urban}}]
	\label{lem:3.2}
	Let $A$ be a ring and $\mathfrak{a}$ be a proper ideal contained in the Jacobson radical of $A$.
	Assume that $A/\mathfrak{a}$ is a domain. Let $L \in A$ be such that its reduction modulo $\mathfrak{a}$ is non-zero.
	Let $I \subseteq (L)$ be an ideal of $A$ and $\overline{I}$ be the image of $I$ in $A/\mathfrak{a}$.
	If $L \pmod{\mathfrak{a}} \in \overline{I}$, then $I = (L)$.
\end{lem}

\begin{proof}
This is \cite[Lem.~3.2]{skinner-urban}.
\end{proof}

For our application, we shall set $A := \Lambda$, $\mathfrak{a} := (\gamma - 1)$ the augmentation ideal of $\Lambda$, $L:= f_{\emptyset, 0}(T)$ a generator of the characteristic ideal of $X_{\emptyset,0}$, and $I$ the ideal generated by $\mathscr{L}_{\mathfrak{p}}^{BDP}$. The divisibility $I\subseteq (L)$ will be a consequence of  Theorem~\ref{thm:howard} and Theorem~\ref{thm:equiv-imc}, and (assuming analytic rank $1$) the relations $0\neq f_{\emptyset,0}(0)\sim_p\mathscr{L}_{\mathfrak{p}}^{BDP}(0)$ will be deduced from Corollary~\ref{cor:structure_of_sha} and Corollary~\ref{cor:equiv-spval}; the equality $I=(L)$ 
%, i.e., Conjecture~\ref{conj:BDP}, 
will then follow.  

\begin{rem}
Note that the roles of algebraic and analytic $p$-adic $L$-functions are switched in our setting comparing with those of \cite{skinner-urban}. This is possible since $\Lambda$ is a UFD, and so the characteristic ideal of a finitely generated $\Lambda$-module is principal.
\end{rem}

\section{Proof of the main results}

We are now ready to prove our  main results.

\begin{thm}\label{thm:mainresult}
Let $E/\mathbf{Q}$ be an elliptic curve of conductor $N$, let  $p>3$ be a good ordinary prime for $E$, and let $K$ be an imaginary quadratic field  of discriminant $D_K$ with $(D_K, Np) = 1$. Assume that: 
\begin{itemize}
	\item $N^-$ is the square-free product of an even number of primes.
	\item $(E,p,K)$ satisfies Condition~$\heartsuit$.
	\item $G_K\rightarrow\mathrm{Aut}_{\mathbf{F}_p}(E[p])$ is surjective.
	\item $p=\mathfrak{p}\overline{\mathfrak{p}}$ splits in $K$.
\end{itemize}
In addition, assume that $\mathrm{ord}_{s=1}L(E/K,s)=1$. Then Conjecture~\ref{conj:HPMC} holds.
\end{thm}

\begin{proof}
By Theorem~\ref{thm:howard}, the Pontryagin dual $X$ of $\mathrm{Sel}_{\mathrm{Gr}}(K,\mathbf{A})$ has $\Lambda$-rank $1$, and its $\Lambda$-torsion submodule $X_{\mathrm{tors}}$ is such that
\[
\mathrm{char}(X_{\mathrm{tors}})\supseteq\mathrm{char}\biggl( \dfrac{ \mathrm{Sel}_{\mathrm{Gr}}(K,\mathbf{T}) }{ \Lambda \kappa^\infty_1 }\biggr)^2.
\]
By Theorem~\ref{thm:equiv-imc}, it follows that the Pontrjagin dual $X^{N^-}_{\emptyset,0}$ of $\mathrm{Sel}^{N^-}_{\emptyset,0}(K,\mathbf{A})$ is $\Lambda$-torsion, and we have
\begin{equation}\label{eq:howard-div}
(f_{\emptyset,0})\supseteq(\mathscr{L}_{\mathfrak{p}}^{BDP})^2,
\end{equation}
where $f_{\emptyset,0}\in\Lambda$ is a generator of the characteristic ideal $\mathrm{char}(X^{N^-}_{\emptyset,0})$.
On the other hand, by the work of Gross--Zagier and Kolyvagin, the assumption that  $\mathrm{ord}_{s=1}L(E/K,s)=1$ implies the Heegner point $y_K\in E(K)$ is non-torsion, and we have  $\mathrm{rank}_{\mathbf{Z}}E(K)=1$ and  $\#\sha(E/K)<\infty$; %Wei Zhang's Theorem~\ref{thm:indivisibility} (see
while by Corollary~\ref{cor:structure_of_sha} we have
\[
[E(K):\mathbf{Z}.y_K]^2\sim_p\#\sha(E/K)[p^\infty].
\] 
In light of Corollary~\ref{cor:equiv-spval}, the last equality (up to a $p$-adic units) amounts to the equality
\begin{equation}\label{eq:equiv-spval}
f_{\emptyset,0}(0)\sim_p\mathscr{L}_{\mathfrak{p}}^{BDP}(0)^2,
\end{equation}   
and given $(\ref{eq:howard-div})$ and $(\ref{eq:equiv-spval})$, the result follows from Lemma~\ref{lem:3.2}.
\end{proof}

\begin{cor}
Let the hypotheses be as in Theorem~\ref{thm:mainresult}. Then Conjecture~\ref{conj:BDP} holds. 
\end{cor}

\begin{proof}
Since $E[p]$ is ramified at all primes $\ell\vert N^-$ by hypothesis, the result follows from Theorem~\ref{thm:equiv-imc} and Theorem~\ref{thm:mainresult}. 
\end{proof}

%We conclude this note with an application of our main result to the converse to the Gross--Zagier--Kolyvagin theorem. 

%\begin{thm}
%Let the hypotheses be as in Theorem~\ref{thm:main}. Then the following implication holds:
%\[
%\mathrm{corank}_{\mathbf{Z}_p}\mathrm{Sel}_{p^\infty}(E/K)=1\;\Longrightarrow\;\mathrm{ord}_{s=1}L(E/K,s)=1.
%\]
%\end{thm}

%\begin{proof}
%Let $(\gamma-1)$ be the augmentation ideal of $\Lambda$. 
%By Mazur's control theorem, the restriction map $H^1(K,E[p^\infty])\rightarrow H^1(K_\infty,E[p^\infty])$ induces a surjective homomorphism
%\[
%X/(\gamma-1)X\longrightarrow\mathrm{Sel}_{p^\infty}(E/K)^\vee
%\]
%with finite kernel, where $\mathrm{Sel}_{p^\infty}(E/K)^\vee=\mathrm{Hom}_{\mathbf{Z}_p}(\mathrm{Sel}_{p^\infty}(E/K),\mathbf{Q}_p/\mathbf{Z}_p)$. %is the Pntrjagin dual.
%Thus by Theorem~\ref{thm:mainresult}, the assumption that $\mathrm{Sel}_{p^\infty}(E/K)^\vee$ has $\mathbf{Z}_p$-rank $1$ implies that $\kappa_1^\infty$ has non-torsion image $\overline{\kappa}_1^\infty$ in the quotient $H^1_{\mathcal{F}_\Lambda}(K,\mathbf{T})/(\gamma_1-1)H^1_{\mathcal{F}_\Lambda}(K,\mathbf{T})$. Since by construction $\overline{\kappa}_1^\infty$ maps to the Kummer image of $y_K$ under the natural injection
%\[
%H^1_{\mathcal{F}_\Lambda}(K,\mathbf{T})/(\gamma_1-1)H^1_{\mathcal{F}_\Lambda}(K,\mathbf{T})\longrightarrow\mathrm{Sel}(K,T),
%\]
%this shows that $y_K$ has infinite order, and the result follows from the Gross--Zagier formula.
%\end{proof}

\bibliographystyle{amsalpha}
\bibliography{bck}

\end{document}